\documentclass[11pt]{amsart}
\usepackage{amsmath}
\usepackage{epsfig}
\usepackage{amsthm}
\usepackage{amssymb}
\usepackage{euscript}

\theoremstyle{plain}
\newtheorem{theorem}{Theorem}[section]
\newtheorem{lemma}[theorem]{Lemma}

\theoremstyle{definition}

\newtheorem{defn}[theorem]{Definition}

\numberwithin{equation}{section}
\usepackage{amsmath, amsthm, amssymb, latexsym, mathrsfs}

\def \mff{\mathsf}

\def \mc{\mathcal}
\def \inv{^{-1}}

\def \v{\vskip 0.1in}

\def \real{\mathbb{R}}

\begin{document}

\title[The $n$-dimensional Abreu's Equation]
{Interior Estimates\\ for the $n$-dimensional Abreu's Equation}

\author[Chen]{Bohui Chen}
\address{Yangtze Center of Mathematics\\
Department of Mathematics\\ Sichuan University\\
Chengdu, 610064, China}
\email{bohui@cs.wisc.edu}
\author[Han]{Qing Han}
\address{Department of Mathematics\\
University of Notre Dame\\
Notre Dame, IN 46556, USA}
\email{qhan@nd.edu}
\address{Beijing International Center for Mathematical Research\\
Peking University\\
Beijing, 100871, China} \email{qhan@math.pku.edu.cn}
\author[Li]{An-Min Li}
\address{Yangtze Center of Mathematics\\
Department of Mathematics\\
Sichuan University\\
 Chengdu, 610064, China}
\email{math\_li@yahoo.com.cn}
\author[Sheng]{Li Sheng}
\address{Department of Mathematics\\
Sichuan University\\
Chengdu, 610064, China}
\email{l\_sheng@yahoo.cn}

\thanks{Chen acknowledges the support of NSFC Grant 11221101.
Han acknowledges the support of NSF
Grant DMS-1105321.
Li acknowledges the support of NSFC Grants 11221101 and 11171235.
Sheng acknowledges the support of NSFC Grants 11101129
and 11201318.}

\begin{abstract}
We study the Abreu's equation in $n$-dimensional polytopes and derive
interior estimates of solutions under the assumption of the
uniform $K$-stability.
\end{abstract}

\maketitle

\section{Introduction}\label{Sec-Intro}

The primary goal of this paper is to study a nonlinear fourth-order partial differential equation
for an $n$-dimensional convex function $u$  of the form
\begin{equation}\label{eqn2.4}
\sum_{i,j=1}^n\frac{\partial^2 u^{ij}}{\partial \xi_i\partial \xi_j}=-A.
\end{equation}
Here, $A$ is a given function and $(u^{ij})$ is the inverse of the Hessian matrix $(u_{ij})$.

The equation \eqref{eqn2.4} was introduced by Abreu \cite{Abreu1998}
in the study of the scalar curvature of toric varieties, in which case the domain
of $u$ is a bounded convex polytope in $\mathbb R^n$
and $A$ is the scalar curvature of toric varieties.
Abreu proved that its solution $u$ yields an extremal metric on toric varieties when
$A$ is a linear function in $\xi$.
Guillemin \cite{Guillemin1994} observed that $u$ is required to have prescribed boundary behavior near
the boundary of the polytope.
As is well known now, the solvability of the equation \eqref{eqn2.4} is closely related to certain
stability conditions.

Tian \cite{T1} first introduced $K$-stability and proved that it is a necessary 
condition for the existence of a K\"ahler-Einstein metric with positive scalar curvature.  
The sufficiency on Fano manifolds was recently established by 
Tian \cite{Tian2013}, and 
by Chen, Donaldson and Sun \cite{CDS2013}. This provides an 
affirmative answer to the Yau-Tian-Donaldson conjecture on Fano manifolds. 


Donaldson \cite{D1} generalized the notion  of $K$-stabililty by giving
an algebro-geometric definition of the Futaki invariant. In particular, he
formulated  $K$-stability for polytopes and  conjectured that
it is equivalent to the existence of K\"ahler metrics of constant scalar
curvature (cscK metrics) on toric varieties.
Donaldson \cite{D2} also considered a
stronger version of stability which we call   {\em uniform
K-stability} in this paper.

Under the assumption of the uniform $K$-stability,
Donaldson \cite{D2} derived interior estimates for
solutions of the Abreu's equation \eqref{eqn2.4}  satisfying Guillemin's boundary conditions
in polytopes in the case of
dimension 2.
Donaldson \cite{D4} subsequently solved \eqref{eqn2.4} when $A$ is constant in the
2 dimensional case, and hence proved the existence of metrics with
constant scalar curvature on 2-dimensional toric varieties.
Recently, Chen, Li and Sheng  \cite{CLS1}, \cite{CLS2} generalized this result and
proved the existence of metrics with prescribed scalar curvature
on 2-dimensional toric varieties.

These works suggest that the uniform $K$-stability is the correct
notion of the stability associated with the existence of
metrics with prescribed scalar curvature on toric varieties.
Indeed, Chen, Li and Sheng \cite{CLS4} proved that the uniform $K$-stability {is} a {\em necessary}
condition of the existence of solutions of \eqref{eqn2.4} satisfying Guillemin's boundary conditions.
It is natural to ask whether such a uniform $K$-stability {is} a {\em sufficient}
condition. Results by Donaldson \cite{D4} and by Chen, Li and Sheng  \cite{CLS1}, \cite{CLS2}
answered this question affirmatively in the 2-dimensional case.

It remains an open problem to study the existence of
metrics with prescribed scalar curvature in higher dimensional
toric varieties.

Recently, there have been several results on the pure PDE aspects of the Abreu's equation.
Feng and Sz\'ekelyhidi \cite{FengSzeke} studied periodic solutions of the Abreu's equation
and proved the existence of a smooth periodic solution of \eqref{eqn2.4}
if $A$ is periodic and has a zero average.
Chen, Li and Sheng \cite{CLS5}  studied the Abreu's equation in bounded, smooth and
strictly convex domains and proved the existence of smooth solutions of \eqref{eqn2.4} for a class of prescribed
boundary values.

In order to relate solutions of \eqref{eqn2.4} to metrics with prescribed scalar curvature on toric varieties,
the equation \eqref{eqn2.4} is required to hold in polytopes and its solutions
satisfy the Guillemin's boundary conditions. This is  a major difficulty associated with \eqref{eqn2.4}.
As the first step of studying the Abreu's equation \eqref{eqn2.4},
we discuss interior estimates of its solutions in polytopes satisfying Guillemin's boundary conditions.
Following Donaldson \cite{D2}, we will keep the differential
geometry in the background.

Before stating the main result in this paper, we first introduce some notations and terminologies.

Let $\Delta$ be a bounded open polytope in $\real^n$, $c_k$ be a constant and
$h_k$ be an affine linear function in $\mathbb R^n$, $k=1, \cdots, K$.
Suppose that
$\Delta$ is defined by linear inequalities $h_k(\xi)-c_k>0$,  for $k=1, \cdots, K$,
where each $h_k(\xi)-c_k=0$ defines a facet of $\Delta$.
Write $\delta_k(\xi)=h_k(\xi)-c_k$
and set
\begin{equation}\label{eqn2.1}
v(\xi)=\sum_k\delta_k(\xi)\log\delta_k(\xi).
\end{equation}
This function was first introduced by Guillemin \cite{Guillemin1994}.
It defines a K\"ahler metric on the toric variety defined by $\Delta$ if $\Delta$ is a Delzant polytope.

We first introduce several classes of functions. Set 
\begin{align*} 
\mc C&=\{u\in C(\bar\Delta):\, \text{$u$ is
convex on $\bar\Delta$ and smooth on $\Delta$}\},\\
\mc S&=\{u\in C(\bar\Delta):\, \text{$u$ is convex  on $\bar\Delta$
and $u-v$ is smooth on $\bar\Delta$}\},\end{align*}
where $v$ is given in \eqref{eqn2.1}. 
For a fixed
point $p_o\in \Delta$, we consider 
\begin{align*}
{\mc C}_{p_o}&=\{u\in \mc C:\, u\geq u(p_o)=0\},\\
\mc S_{p_o}&=\{ u\in \mc S :\, u\geq u(p_o)=0\}.\end{align*}
We say functions in ${\mc C}_{p_o}$ and ${\mc S}_{p_o}$ are {\it normalized} at $p_o$. 

Let $A$ be a  smooth function on $\bar \Delta$.
Consider the  functional
\begin{equation}\label{eqn2.2}
\mc F_A(u)=-\int_\Delta \log\det(u_{ij})d\mu+\mc L_A(u),
\end{equation}
where
\begin{equation}\label{eqn2.3}
\mc L_A(u)=\int_{\partial\Delta}ud\sigma-\int_\Delta Au d\mu.
\end{equation}
When $A$ is a constant, $\mc F_A$ is known to be the Mabuchi functional
and $\mc L_A$ is closely related to the Futaki invariants.

The Euler-Lagrangian equation for $\mc F_A$ is
\begin{equation*}
\sum_{i,j}\frac{\partial^2 u^{ij}}{\partial \xi_i\partial \xi_j}=-A.
\end{equation*}
This is the Abreu's equation \eqref{eqn2.4}.
It is known that,
if $u\in \mc S$ satisfies the equation \eqref{eqn2.4}, then $u$ is an absolute minimizer for
$\mc F_A$ on $\mc S$.




\begin{defn}\label{defn_2.4}
Let $A$ be a smooth function on $\bar\Delta$.
Then, $({\Delta},A)$ is called {\em uniformly $K$-stable} if
the functional $\mc L_A$ vanishes on affine-linear functions and
there exists a constant $\lambda>0$
such that, for any $u\in  {\mc C}_{p_o}$, 
\begin{equation}\label{eqn2.6}
\mc L_A(u)\geq \lambda\int_{\partial \Delta} u d \sigma.
\end{equation}
We also say that $\Delta$ is
$(A,\lambda)$-stable.
\end{defn}

The conditions in Definition \ref{defn_2.4} are exactly the contents of Condition 1
\cite{D2}, introduced by Donaldson.


The following interior estimate is the main result in this paper.

\begin{theorem}\label{theorem_1.3}
Let $\Delta$ be a bounded open polytope in $\real^n$
and $A$ be a smooth function on $\bar\Delta$. 
Suppose  $(\Delta,A)$ is  uniformly $K$-stable and $u$ is a solution in $\mc S_{p_o}$
of  the Abreu's equation \eqref{eqn2.4}.
Then, for any $\Omega\subset\subset \Delta$, any nonnegative integer $k$ and any 
constant $\alpha\in (0,1)$,
\begin{equation*}
\|u\|_{C^{k+3,\alpha}(\Omega)}\leq C\|A\|_{C^k(\bar\Delta)},
\end{equation*}
where $C$ is a positive constant depending only on $n$, $k$, $\alpha$, 
$\Omega$, and $\lambda$
in the uniform $K$-stability.
\end{theorem}


As we mentioned earlier, Donaldson \cite{D2} proved Theorem \ref{theorem_1.3} for $n=2$.
A crucial step in his proof is a derivation of lower and upper bounds of determinants of
the Hessian of
solutions. Donaldson's lower bound holds for all dimensions. However,
his upper bound
is limited to dimension 2.
A major contribution in this paper is
a new upper bound of determinants
of the Hessian in all dimensions. This new upper bound
relates to the Legendre transforms of solutions. Once we have established upper
and lower bounds of determinants of the Hessian, we can prove Theorem \ref{theorem_1.3}
with the help of estimates for linearized Monge-Amp\`{e}re equations due to
Caffarelli and Guti\'{e}rrez \cite{C-G} and estimates for Monge-Amp\`{e}re equations due to
Caffarelli \cite{C1}. Legendre transforms play an important role in our arguments.
In fact, we establish the interior estimates for the Legendre transform of $u$,
instead of
for $u$ directly.


This paper is organized as follows. In Section \ref{sec-Prliminaries}, we
derive an equivalent equation for the Legendre transforms.
In Section \ref{Sec-Determinants}, we derive an upper bound
of the determinants of the Hessian of solutions satisfying
the Guillemin's boundary conditions. Such an upper bound plays an important role in this paper.
Finally in Section \ref{Sec-ProofMain}, we prove Theorem \ref{theorem_1.3}.

\section{Preliminaries}\label{sec-Prliminaries}



\smallskip

In this section, we write the Abreu's equation \eqref{eqn2.4} in its equivalent form for
Legendre transforms.

Let $f=f(x)$ be a smooth and strictly convex function defined in a
convex domain  $\Omega\subset\real^n$.  As $f$ is strictly
convex, $G_f$ defined by
$$
G_f=\sum_{i,j}\frac{\partial^{2}f}{\partial
x_{i}\partial x_{j}}d x_i dx_j= \sum_{i,j} f_{ij}d x_i dx_j
$$
is a Riemannian metric in $\Omega$.
The gradient of $f$
defines a (normal) map $\nabla^f$ from $\mathbb{R}^n$ to $\mathbb{R}^{*n}$:
$$
\xi=(\xi_1,...,\xi_n)=\nabla^f(x) =\left(\frac{\partial f}{\partial
x_1},...,\frac{\partial f}{\partial x_n}\right).
$$
The function $u$ on $\mathbb{R}^{*n}$
$$
u(\xi)=x\cdot\xi - f(x)
$$
is called the {\it Legendre transform} of $f$. We write
$$u=L(f),\;\;\; \Omega^\ast=\nabla^f(\Omega)\subset \mathbb{R}^{*n}.$$
Conversely, $f=L(u).$ It is well-known that $u(\xi)$ is a smooth and strictly convex function.
Corresponding to $u$, we have the metric
$$G_u=\sum_{i,j}\frac{\partial^{2}u}{\partial
\xi_{i}\partial \xi_{j}}d \xi_i d\xi_j=\sum_{i,j} u_{ij}d\xi_id\xi_j.$$
Under the normal map $\nabla^f$, we have
\begin{align*}
\frac{\partial \xi_i}{\partial x_k}&=\frac{\partial^2 f}{\partial x_i\partial x_k},\\
\left(\frac{\partial^2 f}{\partial x_i\partial x_k}\right)
&=\left(\frac{\partial^2 u}{\partial \xi_i\partial \xi_k}\right)^{-1},\end{align*}
and
$$\det\left(\frac{\partial^2 f}{\partial x_i\partial x_k}\right)
=\det\left(\frac{\partial^2 u}{\partial \xi_i\partial \xi_k}\right)^{-1}.$$ Then,
$$(\nabla^f)^*(G_u)=\sum_{i,j}\frac{\partial^{2}f}{\partial
x_{i}\partial x_{j}}d x_i dx_j=G_f,$$
i.e., $ \nabla^f: (\Omega, G_f)\to (\Omega^*,
G_u) $ is locally isometric.

Set
\begin{equation}\label{eqn2.z}\rho=[\det(f_{ij})]^{-\frac{1}{n+2}},\end{equation}
and
\begin{equation}\label{eqn2.a}\Phi=\frac{\|\nabla \rho\|^{2}_{G}}{\rho^{2}}.\end{equation}

Now we derive a formula for
the Laplace-Beltrami operator $\Delta$ in terms of $x_1,...,x_n$ and $f$. Recall that
\begin{equation*}
\Delta
=\frac{1}{\sqrt{\det(f_{kl})}}\sum_{i,j}\frac{\partial}{\partial
x_i}\left(f^{ij}\sqrt{\det\left(f_{kl}\right)}\frac{\partial}{\partial
x_j}\right),\end{equation*}
where
$(f^{ij})$ denotes the inverse matrix of $(f_{ij})$ and
$f_{ij}=\frac{\partial^2 f}{\partial x_i\partial x_j}.$ By a direct calculation,  we get
\begin{equation}\label{eqn2.9}
\Delta  = \sum_{i,j} f^{ij}\frac{\partial^2}{\partial x_i
\partial x_j} -\frac{n+2}{2\rho} \sum_{i,j}
f^{ij}\frac{\partial \rho}{\partial x_j} \frac{\partial } {\partial x_i}
 + \sum_{i,j} \frac{\partial f^{ij} }{\partial x_i}
 \frac{\partial  }{\partial x_j}.\end{equation}
Differentiating  the equality $\sum f^{ik}f_{kj} = \delta^i_j$,
we have  $$\sum_{i,k} \frac{\partial f^{ik} }{\partial x_i} f_{kj} = -
\sum_{i,k} f^{ik} \frac{\partial f_{kj}}{\partial x_i} =
\frac{(n+2)}{\rho} \frac{\partial \rho }{\partial x_j}.$$ It
follows that
\begin{equation}\label{eqn2.10}
\sum_i\frac{\partial f^{ik} }{\partial x_i} =
\frac{(n+2)}{\rho}\sum_j f^{jk} \frac{\partial \rho }{\partial
x_j}.\end{equation}
Inserting \eqref{eqn2.10} into \eqref{eqn2.9},  we obtain
(cf. \cite{L-J-2})
\begin{alignat}{1}\label{eqn2.11}\Delta &= \sum_{i,j}
f^{ij}\frac{\partial^2}{\partial x^i
\partial x^j} + \frac{n+2}{2\rho}\,\sum_{i,j}
f^{ij}\frac{\partial \rho}{\partial x^j} \frac{\partial }
{\partial x^i}.
\end{alignat}
In particular,
\begin{equation}\label{eqn2.12}
\Delta f = n + \frac{n+2}{2\rho} \left\langle \nabla\rho,
\nabla f\right\rangle,
\end{equation}
and
\begin{equation}\label{eqn2.13}
\Delta \left(\sum_k x_k^2\right)= 2\sum_k f^{kk} + \frac{n+2}{2\rho}
\left\langle \nabla\rho, \nabla \sum_k
x_k^2\right\rangle,
\end{equation}
where $\langle\cdot, \cdot\rangle$ is with respect to the metric $G_f$. 
Similarly in terms of coordinates $\xi_1,...,\xi_n$, we have
\begin{equation*}
\Delta =
\sum_{i,j} u^{ij}\frac{\partial^2}{\partial \xi_i
\partial \xi_j} - \frac{n+2}{2\rho}\;\sum_{i,j}
u^{ij}\frac{\partial \rho}{\partial \xi_j} \frac{\partial }
{\partial \xi_i},
\end{equation*}
and hence,
\begin{equation*}
\Delta u = n - \frac{n+2}{2\rho}\left\langle \nabla\rho,
\nabla u\right\rangle,\end{equation*}
and
\begin{equation*}
\Delta \left(\sum_k \xi_k^2\right) = 2\sum_i u^{ii} -
\frac{n+2}{2\rho}
 \left\langle \nabla\rho,
\nabla \sum_k \xi_k^2\right\rangle,\end{equation*}
where $\langle\cdot, \cdot\rangle$ is with respect to the metric $G_u$.

\begin{lemma}\label{lemma_2.8} The Abreu's equation \eqref{eqn2.4}
is equivalent to any of the following two equations:
\begin{equation}\label{eqn2.18}
\sum_{i,j} f^{ij}\frac{\partial^{2}}{\partial x_{i}\partial x_{j}} \left(\log
\det\left(f_{kl}\right)\right)=-A,
\end{equation}
and
\begin{equation}\label{eqn2.19} \Delta
\rho=\frac{n+4}{2}\frac{\left\|\nabla\rho\right\|^{2}}{\rho}
+\frac{\rho A}{n+2}, \end{equation} where and later $\Delta$ and
$\|\cdot\|$ are with respect to the metric $G_f$.
\end{lemma}

\begin{proof} The equivalence between
\eqref{eqn2.4} and \eqref{eqn2.18} is well-known.
We now prove the equivalence between \eqref{eqn2.18} and \eqref{eqn2.19}. Note that
$$-\frac{\partial^{2}}{\partial x_{i}\partial x_{j}}\left(\ln \det\left(f_{kl}\right)\right)=
(n+2)\left(\frac{\rho_{ij}}{\rho}-\frac{\rho_{i}\rho_{j}}{\rho^2}\right),$$
where $\rho_{i}=\frac{\partial\rho}{\partial x_{i}}$ and $\rho_{ij}=
\frac{\partial^{2}\rho}{\partial x_{i}\partial x_{j}}.$ It is easy to see
that \eqref{eqn2.18} is equivalent to
$$\frac{\rho A}{n+2}=\sum_{i,j} f^{ij}\rho_{ij}-\frac{1}{\rho}\sum_{i,j} f^{ij}\rho_{i}\rho_{j},$$
which is equivalent to \eqref{eqn2.19} by  \eqref{eqn2.11}.  \end{proof}

\section{Estimates of the Determinant}\label{Sec-Determinants}

Let $u\in \mc S$ be a solution of the Abreu's equation \eqref{eqn2.4}.
In this section, we derive a global upper bound of the determinant of the 
Hessian of $u$. Recall that the classes $\mc S$ and $\mc S_{p_o}$ were introduced in 
Section \ref{Sec-Intro}. 

The following two lemmas were proved by Donaldson \cite{D2}.
Refer to Theorem 5 and Theorem 6 \cite{D2}.

\begin{lemma}\label{lemma_3.1} Suppose that $u\in \mc S$ satisfies
the Abreu's equation \eqref{eqn2.4}. Then,
$$\det (u_{ij})\geq \mff C_1\quad\text{in }\Delta,$$
where $\mff C_1$ is a positive constant depending only on $n$, $\max_{ \bar\Delta}|A|$ and
$\mathrm{diam}(\Delta).$
\end{lemma}

In fact, we can take
$$\mff C_1 = ({4n\inv \max_{ \bar\Delta}|A|
\mathrm{diam}(\Delta)^2})^{-n}.$$
We point out that we can take $\alpha=0$ in Theorem 5 \cite{D2}.

\begin{lemma}\label{lemma_3.2}
Suppose that $u\in \mc S_{p_o}$ satisfies the Abreu's equation \eqref{eqn2.4}. Assume
that the section
$$\bar{S}_u(p_o,C)=\{\xi\in \Delta:\, u(\xi)\leq C\}$$
is compact and that there is a constant $b>0$ such that $$\sum
_{k=1}^n \left(\frac{\partial u}{\partial \xi_k}\right)^2 \leq b
\quad\text{on }\bar{S}_u(p_o,C).$$ Then,
$$\det (u_{ij})\leq \mff C_2\quad\text{in }S_u(p_o,C/2),$$
where $ \mff C_2$ is a positive constant depending on
$n$, $C$ and $b$.
\end{lemma}

Here, we write Theorem 6 \cite{D2} in the form of Lemma \ref{lemma_3.2}
for convenience of applications in this paper.

We point out that Theorem 6 \cite{D2} holds for all dimensions. However,
in Donaldson's application of this result,
additional information on ``modulus of convexity" of $u$ is required.
Such a modulus of convexity was verified only for the 2-dimensional case.
Refer to Section 5 \cite{D2} for details. It is not clear whether the required
modulus of convexity holds for higher dimensions.

In the following, we derive a global estimate for the upper bound of $\det (D^2u)$,
which plays a key role in this paper. This new upper bound relates to the Legendre
transforms of solutions.

For any point $p$ on $\partial \Delta$, there is an affine coordinate
$\{\xi_1,..., \xi_n\}$, such that, for some $1\leq m \leq n$,
a neighborhood $U\subset \bar\Delta$ of $p$  is defined
by $m$ inequalities
$$\xi_1\ge 0,\quad ...,\quad \xi_m\geq 0,$$
with  $\xi(p)=0.$   Then, $v$ in \eqref{eqn2.1} has the form 
$$v=\sum_{i=1}^{m}\xi_i\log \xi_i+\alpha(\xi),$$
where $\alpha$ is a smooth function in $\bar U$.
By Proposition 2 in \cite{D2},  we have the following result.

\begin{lemma}\label{lemma_3.3} There holds
$$\det(v_{ij})= \big[\xi_1\xi_2 ...\xi_m \beta(\xi)\big]^{-1}\quad\text{in }\Delta,$$
where $\beta(\xi)$ is smooth up to the boundary and $\beta(0)=1$.
\end{lemma}

Denote by $d_E(p,\partial \Delta)$ the Euclidean distance from $p$ to $\partial \Delta$.
By Lemma \ref{lemma_3.3}, we have
\begin{equation}\label{eqn3.2}
\det(v_{ij})\leq \frac{C}{[d_E(p,\partial \Delta)]^n}\quad\text{in  } \Delta,
\end{equation}
where $C$ is a positive constant.

Recall that $p_o\in\Delta$ is the point we fixed for $\mc S_{p_o}$. 
Now we choose coordinates $\xi_1,...,\xi_n$ such that $\xi(p_o)=0$. Set
$$x_i=\frac{\partial u}{\partial \xi_i},\;\;\; f=\sum_i x_i\xi_i - u.$$

\begin{lemma}\label{lemma_3.4}
Suppose that $u\in \mc S_{p_o}$ satisfies the Abreu's equation \eqref{eqn2.4}.
Assume, for some positive constants $d$ and $b$,
$$\frac{1+\sum x_i^2}{(d + f)^2}\leq b\quad\text{in }\mathbb R^n.$$
Then,
$$\exp\left\{ -\mff C_3 f \right\}\frac{\det (u_{ij})}{\left(d+f\right)^{2n}}\leq \mff C_4
\quad\text{in }\Delta,$$
where $\mff C_3$ is a positive constant depending only on $n$ and $\Delta$,
and $\mff C_4$ is a positive constant depending only on 
$n$, $d$, $b$ and $\max_{ \bar\Delta}|A|$. 
\end{lemma}

\begin{proof}
Let $v$ be given as in \eqref{eqn2.1}. 
By adding a linear function, we assume that $v$ is also normalized at $p_o$.
Denote $g=L(v)$. By \eqref{eqn3.2}, it is straightforward to check that
there exists a positive constant $C_1$ such that
$$\det(v_{ij})e^{-C_1g}\to 0\quad\text{as }p\to \partial \Delta.$$
Since $u=v+\phi $ for some $\phi\in C^\infty(\bar \Delta)$, then
\begin{equation}\label{eqn3.a}\det(u_{ij})e^{-C_1 f }\to 0\quad\text{as }p\to \partial \Delta.
\end{equation}
Consider, for some constant $\varepsilon$ to be determined,
$$F=\exp\left\{ -C_2 f + \varepsilon\frac{1+\sum x_i^2}{(d + f)^2}\right\}
\frac{\rho}{\left(d+f\right)^{\frac{2n}{n+2}}},$$
where $C_2=\frac{C_1}{n+2}$ and $\rho$ is defined in \eqref{eqn2.z}. 
By \eqref{eqn3.a}, $F\to 0$ as $p\in\partial\Delta$. 
Assume $F$ attains its maximum at an interior point $p^*$.
Then at $p^*$, we have
\begin{equation}\label{eqn3.3}
\frac{\rho_{,i}}{\rho} - C_2f_{,i}-\frac{2n}{n+2}\frac{f_{,i}}{d+f}
+ \varepsilon\frac{1+\sum x_i^2}{(d + f)^2}\left[\frac{(\sum x_i^2)_{,i}}{1+\sum x_i^2}-
2\frac{f_{,i}}{d + f}\right]=0,\end{equation}
and
\begin{align}\label{eqn3.4}\begin{split}
&\frac{n+2}{2}\Phi + \frac{A}{n+2}- C_2\Delta f - \frac{2n}{n+2}\frac{\Delta f}{d+f}
+ \frac{2n}{n+2}\frac{||\nabla f||^2}{(d+f)^2}\\
&\quad+\varepsilon\frac{1+\sum x_i^2}{(d + f)^2}
\bigg[\frac{\Delta(\sum x_i^2)}{1+\sum x_i^2}-\frac{||\nabla\sum x_i^2||^2}{(1+\sum x_i^2)^2}
-\frac{2\Delta f}{d + f} + \frac{2||\nabla f||^2}{(d + f)^2}\bigg]\\
&\quad+\varepsilon\frac{1+\sum x_i^2}{(d + f)^2}\left(\frac{(\sum x_i^2)_{,i}}{1+\sum x_i^2}-
\frac{2f_{,i}}{d + f}\right)^2
\leq 0,\end{split}\end{align}
where we used \eqref{eqn2.19} 
for $\Delta\rho$. 
By \eqref{eqn2.12} and \eqref{eqn2.13}, we get
\begin{align}\label{eqn3.5}\begin{split}
&\varepsilon\frac{1+\sum x_i^2}{(d + f)^2}
\bigg[\frac{2\sum f^{ii}}{1+\sum x_i^2}-\frac{4\langle\nabla \sum x_i^2,\nabla f\rangle}{(1+\sum x_i^2)(d+f)}
-\frac{2n}{d + f} + \frac{6||\nabla f||^2}{(d + f)^2}\bigg]\\
&\,
-\frac{n+2}{2}\sum \frac{\rho_{,i}}{\rho}\bigg[C_2f_{,i}+\frac{2n}{n+2}\frac{f_{,i}}{d+f}\\
&\qquad\qquad\qquad\qquad
- \varepsilon\frac{1+\sum x_i^2}{(d + f)^2}\left(\frac{(\sum x_i^2)_{,i}}{1+\sum x_i^2}-
2\frac{f_{,i}}{d + f}\right)\bigg]\\
&\,
+\frac{n+2}{2}\Phi+ \frac{2n}{n+2}\frac{||\nabla f||^2}{(d+f)^2}- \frac{2n^2}{n+2}\frac{1}{d+f}+\frac{A}{n+2}
- C_2n\leq 0.\end{split}\end{align}
By inserting \eqref{eqn3.3} into \eqref{eqn3.5} and by the definition of $\Phi$ in \eqref{eqn2.a}, we obtain
\begin{align}\label{eqn3.6}\begin{split}&\varepsilon\frac{1+\sum x_i^2}{(d + f)^2}
\bigg[\frac{2\sum f^{ii}}{1+\sum x_i^2}-\frac{4\langle\nabla \sum x_i^2,\nabla f\rangle}{(1+\sum x_i^2)(d+f)}
-\frac{2n}{d + f} + \frac{6||\nabla f||^2}{(d + f)^2}\bigg]\\
&\quad
+ \frac{2n}{n+2}\frac{||\nabla f||^2}{(d+f)^2}- \frac{2n^2}{n+2}\frac{1}{d+f}
+\frac{A}{n+2}- C_2n\leq 0.\end{split}\end{align}
By the Schwarz inequality, we have
\begin{equation*}
\left|\frac{4\langle\nabla \sum x_i^2,\nabla f\rangle}{(1+\sum x_i^2)(d+f)}\right|
\leq  \frac{||\nabla\sum x_i^2||^2}{4(1+\sum x_i^2)^2} + \frac{16||\nabla f||^2}{(d + f)^2}.\end{equation*}
Hence, 
\begin{equation}\label{eqn3.7}
\left|\frac{4\langle\nabla \sum x_i^2,\nabla f\rangle}{(1+\sum x_i^2)(d+f)}\right|
\leq  \frac{\sum f^{ii}}{(1+\sum x_i^2)^2} + \frac{16||\nabla f||^2}{(d + f)^2}.\end{equation}
Combining \eqref{eqn3.6} and \eqref{eqn3.7} yields
\begin{align*}&\varepsilon\frac{1+\sum x_i^2}{(d + f)^2}
\bigg[\frac{\sum f^{ii}}{1+\sum x_i^2}
-\frac{2n}{d + f} - \frac{10||\nabla f||^2}{(d + f)^2}\bigg]\\
&\quad
+ \frac{2n}{n+2}\frac{||\nabla f||^2}{(d+f)^2}- \frac{2n^2}{n+2}\frac{1}{d+f}
+\frac{A}{n+2}- C_2n\leq 0.\end{align*}
By choosing $\varepsilon>0$ such that $10\varepsilon b\leq 1$, we have 
$$\varepsilon\frac{\sum f^{ii}}{(d + f)^2}+ \frac{A}{n+2}- C_3\leq 0.$$
By the relation between the geometric mean and the arithmetic mean, we get 
$$\frac{\rho}{\left(d+f\right)^{\frac{2n}{n+2}}} 
=\frac{(\det(f^{ij}))^{\frac{1}{n+2}}}{\left(d+f\right)^{\frac{2n}{n+2}}}\leq C_4.$$
Therefore, $F(p^*)\le C_5$, and hence $F\le C_5$ everywhere. The definition of $F$ implies 
$$\exp\left\{ -C_2 f \right\}\frac{\rho}{\left(d+f\right)^{\frac{2n}{n+2}}}\leq C_5.$$
This is the desired estimate. \end{proof}

\section{Proof of the Main Theorem}\label{Sec-ProofMain}

In this section, we prove Theorem \ref{theorem_1.3}.
We first introduce a notation.

For any $u\in \mc C_{p_0}$,  we set
\begin{equation}\label{eqn2.7}
 \|u\|_b=\int_{\partial \Delta} ud\sigma.
\end{equation}
An important consequence of
the uniform $K$-stability is the following result.
See Corollary 2  \cite{D2}.

\begin{lemma}\label{prop_2.7}
Suppose $(\Delta,A)$ is uniformly $K$-stable
and $u\in \mc S_{p_o}$ is a solution of the Abreu's equation \eqref{eqn2.4}. Then,
$$
\|u\|_b\leq C,
$$
where $C$ is a positive constant depending only on $n$ and
$\lambda$.
\end{lemma}

Donaldson pointed out that Lemma \ref{prop_2.7} implies
interior gradient estimates of solutions of the Abreu's equation.
See Corollary 3  \cite{D2}.
As a consequence, a sequence of normalized solutions $\{u^{(k)}\}\subset
{\mc S}_{p_o}$
with uniformly bounded $\|u^{(k)}\|_b$
is locally uniformly convergent to a convex function $u$ in $\Delta$.
(See Section 5 \cite{D1})



Now, we are ready to prove the following result.

\begin{theorem}\label{theorem_4.1}
Suppose that $(\Delta,A)$ is  uniformly $K$-stable and that
$\{A^{(k)}\}$ is a sequence of smooth functions in $\bar\Delta$ such that
$A^{(k)}$ converges to $A$ smoothly in $\bar \Delta$.
Assume $u^{(k)}\in \mc S_{p_o}$ is a sequence of solutions of
the Abreu's equation
\begin{equation}\label{eqn4.1}
\sum_{i,j}\frac{\partial^2(u^{(k)})^{ij}}{\partial\xi_i\partial\xi_j}=-A^{(k)}\quad\text{in }\Delta.
\end{equation}
Then there is a subsequence, still denoted by $u^{(k)}$, such that
$u^{(k)}$ converges to $u$ smoothly in any compact set $\Omega\subset \Delta$,
for some smooth and strictly convex function  $u$ in $\Delta$.
\end{theorem}

We note that Theorem \ref{theorem_4.1} is equivalent to Theorem \ref{theorem_1.3}.

\begin{proof}[Proof of Theorem \ref{theorem_4.1}]
Since $(\Delta,A)$ is uniformly $K$-stable and $A^{(k)}$ converges to $A$ smoothly in $\bar \Delta$,
then $(\Delta,A_k)$ is uniformly $K$-stable for large $k$, i.e., 
$\Delta$ is $(A_k,\lambda)$-stable
for some constant $\lambda>0$ independent of $k$.
Since $u^{(k)}$ satisfies the Abreu's equation \eqref{eqn4.1}, then
$$\mc L_{A_k}(u^{(k)})=\int_{\Delta}\sum_{i,j} (u^{(k)})^{ij}(u^{(k)})_{ij}d\mu=
n \mbox{Area}(\Delta),$$ and hence,
$$
\|u^{(k)}\|_b\leq \lambda^{-1}\mc L_{A_k}(u^{(k)})=\lambda^{-1}n \mbox{Area}(\Delta).
$$
It follows that $u^{(k)}$ locally and
uniformly converges to a convex function $u$ in $\Delta$. 
\v

 {\it Claim.} For any point $\xi\in \Delta$ and any
$B_{\delta}(\xi)\subset \Delta$, there exists a point
$\xi_o\in B_{\delta}(\xi)$ such that $u$ has second derivatives
and is strictly convex at $\xi_o$. Here, $B_{\delta}(\xi)$ denotes the
Euclidean ball centered at $\xi$ with radius $\delta.$

\v

The proof of the claim is the same as in \cite{CLS5}.
For convenience, we present the proof here.
Since $u$ is convex, it has second order derivatives
almost everywhere. Let $G \subset B_{\delta}(\xi)$ be the set
where $u$ has second order derivatives.
Then, $| B_{\delta}(\xi)\setminus G| = 0$. Let $O$ be an open subset of $B_{\delta}(\xi)$
such that $ B_{\delta}(\xi)\setminus G
\subset O$ with $|O|\leq \epsilon$. We choose $\epsilon$ so
small that
$$|B_{\delta}(\xi)\setminus O| >
\frac{1}{2}|B_{\delta}(\xi)|.$$ By Lemma \ref{lemma_3.1} and the weak convergence
of Monge-Amp\`ere measures, we have
\begin{equation}\label{eqn_4.2}
\int_{B_{\delta}(\xi)\setminus O}\det(u_{kl})d\mu > \frac{1}{2}\mff C_1
|B_{\delta}(\xi)|.
\end{equation}
Hence,  there exists a
point $\xi_o \in B_{\delta}(\xi)\setminus O$ such that
$$\det(u_{kl})(\xi_o) \geq
\frac{\mff C_1|B_{\delta}(\xi)|}{2|B_{\delta}(\xi)\setminus O|}.$$
The claim is proved.

\v
We now choose coordinates such that $\xi_o=0$.
By adding linear functions, we assume that all $u^{(k)}$ and $u$ are normalized at $0$.
Since $u$ is strictly convex at $0$, there exist constants  $\epsilon'>0$, $d_2>d_1>0$ and $b' >0$,
independent of $k$, such that, for large $k$,
$$B_{d_1}(0)\subset\bar{S}_{u^{(k)}}(0,\epsilon')\subset B_{d_2}(0)\subset \Delta,$$ and
$$
\sum_i \left(\frac{\partial u^{(k)}}{\partial \xi_i}\right)^2\leq b'
\quad\text{in } S_{u^{(k)}}(0,\epsilon').$$  By
Lemma \ref{lemma_3.1} and Lemma \ref{lemma_3.2}, we have
\begin{equation}\label{eqn4.3}
C_1\leq \det(u_{ij}^{(k)}) \leq C_2
\quad\text{in }S_{u^{(k)}}(0,\frac{1}{2}\epsilon'),\end{equation}
where $C_1<C_2$ are positive constants independent of $k$.

By an estimate due to Caffarelli and Guti\'{e}rrez  \cite{C-G}, 
there is a uniform interior $C^\alpha$-bound of  
$\det(u^{(k)}_{ij})$.  Caffarelli and Guti\'{e}rrez originally proved this result 
for homogeneous linearized Monge-Amp\`ere equations. 
Trudinger and Wang   \cite{TW2002} pointed out that such a result can be extended 
to  the Abreu's equation \eqref{eqn4.1} if $A_k\in
L^\infty(\bar \Delta)$, under the assumption \eqref{eqn4.3}. 
(See also \cite{D2}.)  By  $C^{2,\alpha}$
estimates for Monge-Amp\`ere equation due to Caffarelli \cite{C1},
we have, for any $\Omega^*\subset B_{d_1}(0)$,
$$\|u^{(k)}\|_{C^{2,\alpha}(\Omega^*)}\leq C_2.$$
Then, we employ the Schauder estimate to conclude that $\{u^{(k)}\}$
converges smoothly to $u$. Therefore,
$u$ is  a smooth and
strictly convex function in $S_u(0,\epsilon'/2)$.

Let $f^{(k)}$ be the Legendre transform of $u^{(k)}$.
Then, $\{f^{(k)}\}$ locally
uniformly converges to a convex function $f$ defined in the
whole $\mathbb{R}^n$. Furthermore, in a neighborhood of $0$,
$f$ is a smooth and strictly convex function such that its
Legrendre transform $u$ satisfies the Abreu's equation.

By the convexity of $f^{(k)}$ and the local and uniform convergence of
$\{f^{(k)}\}$ to $f$, we conclude, for any $k$,
$$\frac{1+\sum_i x_i^2}{(d+ f^{(k)})^2} \leq b\quad\text{in }\mathbb R^n,$$
and, for any $C>1$,
$$B_r(0)\subset S_{f^{(k)}}(0,C)\subset B_{R_C}(0),$$
for some positive  constants $d$, $b$, $r$ and $R_C=R(C)>0$.

By Lemma \ref{lemma_3.1} and Lemma \ref{lemma_3.4}, we have
$$\exp\{-\mff C_3 C\}\frac{1}{(d+C)^{2n}}\leq \det (f_{ij}^{(k)})\leq \mff C_1 .$$
We note that each $f^{(k)}$ satisfies \eqref{eqn2.18}, with $f$ and $A$ 
there replaced by $f^{(k)}$ and 
$A^{(k)}$. By the $C^\alpha$-estimates due to Caffarelli and Guti\'{e}rrez  and the
$C^{2,\alpha}$-estimates due to Caffarelli as above,
we conclude that $\{f^{(k)}\}$ uniformly and smoothly converges to $f$
in $S_f(0,C/2)$. Since $C$ is arbitrary, $f$ is a smooth and strictly convex function in $\mathbb{R}^n$,
and the sequence $\{f^{(k)}\}$ locally and smoothly converges to $f$. By Legendre transforms,
we obtain that $u$ is a smooth and strictly convex function in $\Delta$
and that the sequence $\{u^{(k)}\}$ locally and smoothly converges to $u$.
This completes the proof of Theorem \ref{theorem_4.1}. \end{proof}

\end{document}